\documentclass[reqno]{amsart}

\usepackage[margin=1in]{geometry}
\usepackage[english]{babel}
\usepackage[utf8]{inputenc}
\usepackage{amsmath,amssymb,amsthm, comment,enumerate,mathtools}
\usepackage{hyperref}
\mathtoolsset{showonlyrefs}

\newcommand{\pa}{\partial}
\newcommand{\D}{\mathcal{D}}
\newcommand{\E}{\mathcal{E}}
\newcommand{\Vol}{\textrm{Vol }}
\newcommand{\dn}{\mathcal{N}} % Neumann map
\newcommand{\har}{\mathcal{H}} % Harmonic extension
\newcommand{\A}{\mathcal{A}} % some bounded quantites
\newcommand{\tn}{\overline{\nabla}} % boundary derivative
\newcommand{\R}{\mathbb{R}}
\newcommand{\N}{N} % normal to the boundary
\newcommand{\K}{\mathcal{K}} % Biot-Savart kernel
\newcommand{\ve}{\varepsilon}
\newcommand{\dist}{\textrm{dist}_g}

\let\div\relax
\DeclareMathOperator{\curl}{curl}
\DeclareMathOperator{\div}{div}

\numberwithin{equation}{section}

\newtheorem{theorem}{Theorem}
\newtheorem{prop}{Proposition}
\newtheorem{cor}{Corollary}
\newtheorem{lemma}{Lemma}
\theoremstyle{remark}

\theoremstyle{definition}

\title{On the breakdown of solutions to the incompressible Euler equations
with free surface boundary}

\author{Daniel Ginsberg}\address{Johns Hopkins University, Department of Mathematics, 3400 N.\@ Charles St., Baltimore, MD 21218, USA}
\email{dginsbe5@math.jhu.edu}

\date{\today}

\begin{document}
\maketitle
\vspace{-0.25in}
\nocite{*}
\begin{abstract}
  We prove a continuation critereon for incompressible liquids with free
  surface boundary. We combine the energy estimates of Christodoulou and
  Lindblad with an analog of the estimate due to Beale, Kato, and Majda for the
  gradient of the velocity in terms of the vorticity, and use this to show
  solution can be continued so long as the second fundamental form and
  injectivity radius of the free boundary, the vorticity, and
 one derivative of the velocity on the free boundary
remain bounded, assuming that the Taylor sign condition holds.
\end{abstract}
\section{Introduction}

We consider the equations of motion for an incompressible, bounded
fluid body with velocity $u = (u_1, u_2, u_3)$ and pressure $p$:
\begin{alignat}{2}
 (\pa_t + u^k\pa_k)u_i &= -\pa_i p, &&\quad i = 1,.., 3,  \textrm{ in } \D,\label{mom}\\
 \div u = \pa_i u^i &= 0, && \quad\textrm{ in } \D.
 \label{mass}
\end{alignat}
Here, we are summing over repeated upper and lower indices and writing
$\pa_t = \frac{\pa}{\pa t}$ and $\pa_k = \frac{\pa}{\pa x^k}$
as well as
$u^i = \delta^{ij}u_j$.
The domain $\D = \bigcup_{0\leq t \leq T} \{t\} \times \D_t$
is to be determined, and satisfies:
\begin{alignat}{3}
 p = 0 \textrm{ on } \pa \D_t, &&\quad
 p > 0 \textrm{ in } \D_t,&&\quad
 (\pa_t + u^k\pa_k) \big|_{\pa \D} \in
 T(\pa \D),
 \label{bc}
\end{alignat}
with $\D_0$ diffeomorphic to the unit sphere.

The problem \eqref{mom}-\eqref{bc} is ill-posed in Sobolev spaces unless the
following condition holds, known as the ``Taylor sign condition'' (see \cite{Ebin1987}):
\begin{alignat}{3}
  (-\nabla_N p)(t,\cdot) > 0 &&\quad \textrm{ on } \pa \D_t
  &&\quad \text{ where } \nabla_N = N^i \nabla_i.
 \label{tsc}
\end{alignat}

This and related problems have been studied by a wide variety of authors. In the
case that $\D_0$ is diffeomorphic to the lower half plane and the vorticity
$\omega = \curl v$ vanishes, this is known as the irrotational water waves problem.
The literature on this
problem is vast; let us just single out \cite{Germain2009}, \cite{Wu2010}
and \cite{Ionescu2016}, where the authors proved that in two and three dimensions, the
water waves problem is globally well-posed in Sobolev spaces.
See \cite{Ionescu2018} for a comprehensive survey of this problem. See also
\cite{Ifrim2015}, where the authors considered the water waves problem
in two dimensions with constant vorticity and proved that the system
has a cubic lifespan, and \cite{Bieri2017}, where the authors
considered the free boundary problem in two dimensions with self-gravitation
but no vorticity and also proved a cubic lifespan bound.

In the case that $\D_0$ is a bounded domain and $\omega \not = 0$,
Lindblad and Christodoulou
proved energy estimates in \cite{Christodoulou2000}. Local
well-posedness was then shown in this case by Lindblad in \cite{Lindblad2003}
using a Nash-Moser iteration and later by Shkoller and Coutand in \cite{CS07}
using a tangential smoothing operator.

We now describe the energy estimates developed by Lindblad and Christodoulou.
We let $\N_i$ denote the unit conormal to $\pa \D_t$ and
let $\Pi_{i}^j = \delta_i^j - \N_i \N^j$ denote the
projection to the tangent space at the boundary.
Writing
$\nabla$ for the covariant derivative on $\D_t$,
we let $\theta$ denote the second fundamental form of $\pa \D_t$:
\begin{equation}
 \theta_{ij} = \Pi_{i}^k \Pi_j^\ell \nabla_k N_\ell
 \label{}
\end{equation}

The energies in \cite{Christodoulou2000} are of the form:
\begin{align}
 \E_r(t) = \int_{\D_t} \delta^{ij} Q(\nabla^r u_i, \nabla^r u_j)
 + \int_{\pa\D_t} Q(\nabla^r p, \nabla^r p) |\nabla p|\, dS
 + \int_{\D_t} |\curl \nabla^{r-1} u|^2,
 \label{enint}
\end{align}
where $Q$ is a quadratic form which is the norm of the tangential
components, $Q(\beta,\beta) = |\Pi \beta|^2$ when restricted to the boundary,
and the full norm in the interior.

Let $\iota_0$ denote the injectivity radius of the normal exponential
map. By definition this is the smallest number $\iota$ so that the map:
\begin{equation}
 (x,\iota')
  \to  x + \iota' N(x) \in \D_t
 \label{}
\end{equation}
defined for $x \in \pa \D_t, \iota' \in (-\iota,\iota)$.
is injective.  Thus $\iota_0$ measures how far $\pa \D_t$
is from self-intersecting.
The main theorem in \cite{Christodoulou2000} is:
\begin{theorem}
  \label{LCthm}
  Let $(u, \D)$ be a smooth solution to \eqref{mom}-\eqref{bc}
  for $0 \leq t \leq T$.
 Suppose that the following bounds hold:
 \begin{alignat}{2}
  |\nabla u| + |\nabla p| &\leq M &&\quad \textrm{ in } \D_t,\label{bl1}\\
  |\theta| + \frac{1}{\iota_0} &\leq K, &&\quad \textrm{ on } \pa \D_t,\label{bl2}\\
  -\nabla_Np \geq \delta &> 0 , &&\quad \textrm{ on } \pa \D_t,\label{bl3}\\
  |\nabla^2 p| + |\nabla D_t p| &\leq L &&\quad \textrm{ on } \pa \D_t.
  \label{bl4}
\end{alignat}
 Then there are continuous functions $F_r$ with $F_r|_{t = 0} = 1$ so that:
 \begin{align}
  \E_r(t) \leq F_r(t, M, K, \delta^{-1}, L, \E_0(0),..., \E_{r-1}(0))
  \E_r(0)
  \label{}
 \end{align}
\end{theorem}

By the local well-posedness results in e.g. \cite{L05a} or \cite{CS07}, Theorem
\ref{LCthm} and an approximation argument to go from smooth solutions
to solutions in Sobolev spaces, this gives the following breakdown critereon for solutions to
\eqref{mom}-\eqref{bc}:
\begin{cor}
  Let $(u,\D)$ be a solution to \eqref{mom}-\eqref{bc}
  with:
  \begin{alignat}{3}
   u(t) \in H^s(\D_t), &&\quad \D_t \in H^{s+1/2}(\R^3), &&\quad 0 \leq t \leq T.
   \label{uureg}
 \end{alignat}
  for $s \geq 3$. Suppose that
  $T^*$ is the largest time so that $u$ can be continued as a solution to
  \eqref{mom}-\eqref{bc} in the class \eqref{uureg}. Then either
  $T^* = \infty$ or at least one of the quantities on the left-hand sides of
  \eqref{bl1}, \eqref{bl2}, \eqref{bl4} go to infinity,
  or the quantity on the left-hand side of \eqref{bl3} goes to
  zero as $t \nearrow T^*$.
\end{cor}
Here, $\D_t \in H^{s+1/2}(\R^3)$ means that locally, $\D_t$ can be written
as the graph of a function in $H^{s+1/2}(\R^3)$.
Our result is that the above breakdown condition can be replaced with a condition
on the vorticity $\omega$ and some norms of $u$ on the free boundary $\pa \D_t$.
Our approach follows the seminal
article \cite{Beale1984}, where Beale, Kato, and Majda
consider
the incompressible Euler equations in $\R^3$:
\begin{alignat}{2}
 (\pa_t  + u^k\pa_k)u_i &= -\pa_i p, && \quad \textrm { in } [0,T]\times \R^3,
 \label{u3}\\
 \div u &= 0, && \quad \textrm{ in } [0,T] \times \R^3.
 \label{u32}
\end{alignat}
with solutions in the following space:
\begin{equation}
 u \in C([0,T]; H^3(\R^3)) \cap C^1([0,T]; H^{2}(\R^3)).
 \label{ureg}
\end{equation}
The show that if $T^* < \infty$ is the first time for which the solution to \eqref{u3}-\eqref{u32}
cannot be continued beyond $T = T^*$ in the class \eqref{ureg}, then:
\begin{align}
 \int_0^{T^*} ||\curl v(t, \cdot)||_{L^\infty(\R^3)} = \infty
 \label{bkm}
\end{align}
 See also
\cite{Ferrari1993} for a generalization to the case of a fixed, bounded
domain with a Neumann boundary condition.

We now describe our main theorem.
We write $\dn$ for the Dirichlet-to-Neumann operator on $\pa \D_t$;
if $\psi :\pa\D_t \to \R$ and $\psi_\har: \D_t \to \R$ denotes the harmonic
extension of $\psi$ to $\D_t$, then:
\begin{equation}
 \dn \psi = \big(\N\cdot \nabla \psi_\har\big)\big|_{\pa \D_t}.
 \label{}
\end{equation}
We will also write $U = u|_{\pa \D_t}$.
Our main result is then:
\begin{theorem}
  \label{mainthm}
  Let $u$ be a solution to \eqref{mom}-\eqref{mass} in the class
  \eqref{uureg} with $s > 3$. Define:
  \begin{align}
   \A(t) &= ||\omega(t)||_{L^\infty(\D_t)} + ||\nabla u(t)||_{L^\infty(\pa\D_t)}
   + ||\dn U(t)||_{L^\infty(\D_t)},\\
   \mathcal{K}(t) &= ||\theta(t)||_{L^\infty(\pa \D_t)} + \frac{1}{\iota_0(t)}
   + ||(\nabla_N p(t))^{-1}||_{L^\infty(\pa \D_t)}.
   \label{}
  \end{align}
  Suppose that there is a time $T^*$ so that
  the solution cannot be continued past $T^*$
  in the class \eqref{uureg} and that $T^*$ is the
  first such time. Then either:
  \begin{align}
   \limsup_{t \nearrow T^*}\, \mathcal{K}(t) = \infty,
   \label{}
  \end{align}
  or
  \begin{align}
   \int_0^{T^*}  \big(\A(t)\big)^2
   + ||\nabla_N D_t p(t)||_{L^\infty(\pa\D_t)} \, dt = \infty.
   \label{blowup}
  \end{align}
  In particular, if \eqref{blowup} occurs, then:
  \begin{align}
   \limsup_{t \nearrow T^*} \A(t) + ||\nabla_N  D_tp(t)||_{L^\infty(\pa\D_t)} = \infty.
   \label{pwblow}
  \end{align}
\end{theorem}

The reason that this criterion requires a bound for $\A^2$ instead of
$\A$ is ultimately
due to the fact that the energy estimates we will use require a bound for
$D_t p$ in the interior. Since $\Delta p = -(\nabla^i u_j)(\nabla^j u_i)$,
it follows that
$\Delta D_t p$ is cubic in $u$ (see Lemma \ref{dtpests}).

As in \cite{Beale1984} and \cite{Ferrari1993}, the proof relies on two
ingredients. First, in Section \ref{enestsec} we assume that $(u,\D)$
is a solution to \eqref{mom}-\eqref{bc} in the class
 \eqref{uureg} and that the following bounds hold:
\begin{alignat}{2}
  |\theta(t)| + \frac{1}{\iota_0(t)}
  &\leq K, &&\quad \textrm{ on } \pa\D_t,
  0 \leq t \leq T,\label{assump1}\\
  (-\nabla_N p(t)) \geq \delta &> 0, &&\quad \textrm{ on } \pa \D_t.
  0 \leq t \leq T,
 \label{assump2}
\end{alignat}
Then, with $\E_r$ defined by \eqref{enint}:
\begin{equation}
 \frac{d}{dt} \E_3(t) \leq  C(K,\delta^{-1}) (||\nabla u||_{L^\infty(\D_t)}
 + \A(t) + \big(\A(t)\big)^2 + ||\nabla_N D_t p||_{L^\infty(\pa \D_t)}) \E_3(t),
 \label{enest1}
\end{equation}
and for $s \geq 4$,
\begin{align}
 \frac{d}{dt} \E_s(t) \leq C(K, \delta^{-1}) (||\nabla u||_{L^\infty(\D_t)}
 + \A(t) + \big(\A(t)\big)^2 + ||\nabla_N D_t p||_{L^\infty(\pa \D_t)}) \big(
 \E_s(t) + P(\E_{s-1}(t), ..., \E_0(t))\big).
 \label{enest2}
\end{align}

Write $u = u_0 + u_1$ where:
  \begin{alignat}{2}
   \Delta u_0 &= \curl \omega, &&\quad \textrm{ in } \D_t,\label{u0def}\\
   u_0 &= 0, &&\quad \textrm{ on } \pa \D_t,
   \label{u0def2}
 \end{alignat}
  and $\Delta u_1 = 0$ in $\D_t$, $u_1|_{\pa \D_t} = u|_{\pa \D_t}$.
In section \ref{nlinsec}, we prove the following nonlinear estimate:
\begin{equation}
 ||\nabla u_0(t)||_{L^\infty(\D_t)}
 \leq C(K) \big( (1 + \log^+(||\omega(t)||_{H^2(\D_t)}) ||\omega(t)||_{L^\infty(\D_t)} + 1 \big),
 \label{nlinintro}
\end{equation}
with $\log^+(s) = \max(0, \log(s))$, which relies on properties of the Green's
function for the Dirichlet problem in $\D_t$.

Assuming that \eqref{enest1}-\eqref{enest2} and \eqref{nlinintro} hold
for the moment, we have:
\begin{proof}[Proof of Theorem \ref{mainthm}]
  Write $y_s(t) = \E_s(t), y_s^* = \sum_{k \leq s} \E_{k}$.
  By the maximum principle \eqref{bernstein}, we have
  $||\nabla u_1||_{L^\infty(\D_t)} \leq \A(t)$.
  Therefore, by \eqref{enest1} and \eqref{nlinintro}, we have
  that $y_3(t)$ satisfies the differential inequality:
  \begin{equation}
   \frac{d}{dt} y_3 \leq C(K) \bigg( \A + \A^2
  + ||\nabla_N D_t p||_{L^\infty(\pa \D_t)} \bigg)
   y_3 (1 + \log^+ y_3)
   \label{en1}
  \end{equation}
  and that for $s \geq 4$:
  \begin{equation}
   \frac{d}{dt} y_s \leq C(K) \big(\A + \A^2 +
   ||\nabla_N D_t p||_{L^\infty(\pa \D_t)})
   \big(y_s (1 + \log^+ y_s) + P(y_{s-1}^*)\big).
   \label{}
  \end{equation}

  As in the proof of Proposition 17.2.3 in \cite{Taylor2011}, we now note that
  \eqref{en1} implies:
  \begin{equation}
   \limsup_{t \nearrow T^*} \int_{y_3(0)}^{y_3(t)} \frac{dy}{y (1+ \log^+ y)}
   \leq C(K)\, \limsup_{t \nearrow T^*} \int_0^{t} \A(s) + \big(\A(s)\big)^2
   + ||\nabla_N D_t p(s)||_{L^\infty(\pa \D_s)}\, ds
   \label{contra}
  \end{equation}
  and this is finite, by assumption.

  On the other hand, since $T_*$ is the largest time for which \eqref{mom}-\eqref{bc}
  has a solution in the space \eqref{ureg}, we have $\limsup_{t \nearrow T_*} y_3(t) = \infty$.
  However,
  \begin{equation}
   \lim_{s \to \infty} \int_a^s \frac{dy}{y(1 + \log^+ y)}
   = \infty,
   \label{}
  \end{equation}
  for any $a > 0$,
  which contradicts the fact that the right-hand side of \eqref{contra}
  is finite.
  The result for $s \geq 4$ then follows from induction and
  \eqref{enest2}.

\end{proof}

\section{The free boundary $\pa \D_t$ and the projection to the
tangent space at the boundary}
It is easiest to define the geometric quantities that we will need in terms of
\emph{Lagrangian coordinates}, which we now define. Let $\Omega \subset \R^3$ be diffeomorphic
to the unit ball. We define $x^i(t) = x^i(t,\cdot): \Omega \to \D_t$ by:
\begin{align}
 \frac{d}{dt} x^i(t,y) &= u^i(t, x(t,y)),\\
 x^i(0, y) &= f_0(y),
 \label{}
\end{align}
where $f_0:\Omega \to \D_0$ is a volume preserving diffeomorphism.
This change of coordinates and the metric $\delta_{ij}$ in $\R^3$ induce
a time-dependent metric $g = g(t,y)$ on $\Omega$:
\begin{equation}
 g_{ab}(t,y) = \delta_{ij} \frac{\pa x^i}{\pa y^a}
 \frac{\pa x^j}{\pa y^b}.
 \label{gdef}
\end{equation}
As in \cite{Christodoulou2000}, we will work with the covariant derivative
associated to $g$; if
$\alpha = \alpha_{a_1\cdots a_r} dy^{a_1}\cdots dy^{a_r}$
is a $(0,r)$ tensor then $\nabla \alpha$ is a
$(0, r+1)$ tensor with components:
\begin{equation}
 \nabla_a \alpha_{a_1\cdots a_r} =
 \pa_a \alpha_{a_1\cdots a_r} - \Gamma^b_{a_1 a} \alpha_{b a_2\cdots a_r} - \cdots
- \Gamma^b_{aa_r} \alpha_{a_1 \cdots a_{r-1} b},
 \label{}
\end{equation}
where the Christoffel symbols $\Gamma_{ab}^c$ are defined by:
\begin{equation}
 \Gamma_{ab}^c = \frac{1}{2}g^{cd} \bigg(\frac{\pa}{\pa y^a} g_{bd}
 + \frac{\pa}{\pa y^b} g_{ad} - \frac{\pa}{\pa y^d} g_{ab}\bigg).
 \label{}
\end{equation}
We will frequently make use of the fact that the covariant
derivatives commute, which just follows from the fact that
they commute when expressed in Eulerian coordinates. We will
write
$\nabla^r_{a_1\cdots a_r} = \nabla_{a_1}\cdots
\nabla_{a_r}$ and will often omit the indices $a_1,\cdots a_r$.

We will write $D_t$ for time differentiation in Lagrangian coordinates:
\begin{equation}
 D_t  = \frac{\pa}{\pa t}\bigg|_{y = \textrm{const}}
  = \frac{\pa}{\pa t} \bigg|_{x = \textrm{const}}
  + u^k\frac{\pa}{\pa x^k}.
 \label{}
\end{equation}
Note that $D_t$ does not commute with $x$ derivatives:
\begin{equation}
 [D_t, \pa_i] = -(\pa_i u^k)\pa_k
 \label{commform}
\end{equation}
We will use the convention that the letters $a,b,c...$
refer to quantites expressed in Lagrangian coordinates
and the letters $i,j,k,...$ refer to quantites
expressed in the Eulerian coordinates.

\subsection{The geometry of $\pa \D_t$}
We let $N^a$ denote the unit normal to $\pa \Omega$ with respect to the
metric $g$ defined by \eqref{gdef}, and write $N_a = g_{ab} N^b$ for the
unit conormal. We will write $\gamma$ for the (co)-metric on $\pa \Omega$:
\begin{alignat}{2}
 \gamma_{ab} = g_{ab} - N_a N_b,
 &&\quad
 \gamma^{ab} = g^{ab} - N^aN^b.
 \label{}
\end{alignat}
We also write:
\begin{equation}
 \Pi^a_b = \delta^{a}_b - N^a N_b
 \label{}
\end{equation}
for the orthogonal projection to $T(\pa \Omega)$. More generally
if $\alpha$ is a $(r,s)$ tensor, we set:
\begin{equation}
 (\Pi \alpha)_{a_1\cdots a_s}^{b_1\cdots b_r}
 = \Pi_{a_1}^{c_1}\cdots \Pi_{a_s}^{c_s}
 \Pi_{d_1}^{b_1}\cdots \Pi_{d_r}^{b_r}
 \alpha_{c_1\cdots c_s}^{d_1 \cdots d_r}.
 \label{}
\end{equation}
We will write $\tn$ for covariant differentiation on $\pa \Omega$,
defined by:
\begin{equation}
 \tn \alpha = \Pi \nabla \alpha,
 \label{}
\end{equation}
for an arbitrary tensor field $\alpha$.
The second fundamental form $\theta_{ab}$ is given by:
\begin{equation}
 \theta_{ab} = \tn_a N_b
 = \Pi_a^c \Pi_b^d \nabla_c N_d.
 \label{}
\end{equation}

We let $\iota_0$ denote the injectivity radius of $\pa \D_t$.
The fundamental geometric assumption that we will make is that:
\begin{equation}
 |\theta| + \frac{1}{\iota_0} \leq K, \textrm{ on } \pa \D_t.
 \label{geom}
\end{equation}
Among other things, this assumption ensures that the domain
$\D_t$ satisfies the ``uniform exterior sphere condition'':
\begin{lemma}
  \label{bdyreg}
 If \eqref{geom} holds, then
there is a $r_0 = r_0(K)$ with $r_0 > 0$ so that for each
 $x \in \pa \D_t$, there is an $r > r_0$ and balls $B_1, B_2$ with radius
 less than $r$ so that:
 \begin{alignat}{2}
   B \cap \overline{\D_t} = \{x\},
   &&\quad
   B_2 \cap \overline{\R^3 \setminus \D_t} = \{x\}.
  \label{}
\end{alignat}

\end{lemma}
\begin{proof}
  Fix $x \in \pa \D_t$ and take $r \leq K/2$. With $z_{\pm} = x \pm r N(x)$,
  where $N(x)$ is the normal vector to $\pa \D_t$ at $x$, the ball of
  radius $r$ centered at $z_{\pm}$ touches $x$ but does not cross
  $\pa \D_t$ (from the outside in the ``$+$'' case and the inside in the
    ``-'' case) since the largest principal curvature at $x$ does not
  exceed $K$. Also, by the bound for the injectivity
  radius, there are no other points of $\pa \D_t$ in this ball.
\end{proof}

We now let $d = d(t,y)$ denote the geodesic distance to the boundary and
define the extension of the normal to the interior of $\Omega$ by:
\begin{equation}
 \widetilde{N}_a = \frac{\nabla_a d}{\sqrt{g^{ab}\nabla_a d\nabla_b d}},
 \label{}
\end{equation}
which is well-defined so long as $(t,y)$ are such that $d(t,y) < \iota/2$, say.
From now on we will abuse notation and just write $N$ instead of $\tilde{N}$.
Note that with this choice of $d$, we have $\nabla_N N = 0$ near the boundary.
We now extend this definition to all of $\pa \Omega$. Let $\varphi \in C_0^\infty(\R)$
be so that $\varphi(x) = 1$ when $|x| \leq 1/4$ and $\varphi(x) = 0$ when $|x| \geq 1/2$,
and set $\varphi_{\iota}(x) = \varphi(\iota^{-1} x)$. We then abuse notation
further and write:
\begin{alignat}{2}
 \Pi^a_b = \delta^a_b - \varphi_\iota N^a N_b,
 &&\quad \tn_b = \Pi^a_b \nabla_b.
 \label{piint}
\end{alignat}
Away from the boundary $\Pi$ is just the identity map and close to the boundary
it agrees with the projection onto the tangent space to the level sets of $d$.
Similarly, away from the boundary $\tn$ is the covariant derivative
and on $\pa \Omega$ it is the covariant derivative on $\pa \Omega$.

If the assumption \eqref{geom} holds, we can control derivatives of $\Pi$:
\begin{lemma}
 Let $d(y) = \dist(y, \pa \Omega)$ denote the geodesic distance in the metric
 $g$ from $y$ to $\pa \Omega$. If $n$ denotes the conormal $n = \nabla d$ to the
 sets $ \{ y \in \Omega: d(y) = a\}$ and $\Pi$ is defined by \eqref{piint}, then:
 \begin{align}
  |\nabla n(t,y)|  + |\nabla \Pi(t,y)| \leq C ||\theta(t,\cdot)||_{L^\infty(\pa \Omega)},
  && |D_t n(t,y)| + |D_t \Pi(t,y)| \leq C ||\nabla u(t,\cdot)||_{L^\infty(\Omega)},
  \label{dgammabd}
 \end{align}
 when $d(y) < \frac{\iota_0}{2}$, where $\iota_0$ is the normal injectivity
 radius of $\pa \D_t$.
\end{lemma}
\begin{proof}
 See Lemmas 2.1, 3.10-3.11 in \cite{Christodoulou2000}.
\end{proof}

\section{Elliptic estimates}
\label{ellsec}

We will need many of the elliptic estimates from \cite{Christodoulou2000}. The basic result
we rely on is the following pointwise inequality.
Let $\beta = \beta_{Ii}  \nabla_{I}^r \alpha_i$ for a $(0,1)$-tensor
$\alpha$. We will write:
\begin{alignat}{2}
 \div \beta = \nabla^i \beta_{Ii},
 &&\quad \curl \beta_{ij} = \nabla_i \beta_{Ij} - \nabla_j \beta_{Ii},
 \label{}
\end{alignat}
and:
\begin{equation}
 (\Pi \beta)_{Ji} = \Pi^I_J \beta_{Ii}
 = \Pi^{i_i}_{j_1} \cdots \Pi^{i_s}_{j_s} \beta_{i_1 \cdot i_s i}.
 \label{}
\end{equation}
We then have Lemma 5.5 from \cite{Christodoulou2000}:
\begin{lemma}
  With the above notation:
  \begin{equation}
   |\nabla \beta| \leq C\big( |\div \beta| + |\curl \beta| +
   |\Pi \nabla \beta|\big).
   \label{pwlem}
  \end{equation}
\end{lemma}
We will also use the following simple version of the trace inequality,
which is estimate (5.19) in \cite{Christodoulou2000}.
\begin{lemma}
  If \eqref{geom} holds, then:
  \begin{equation}
   ||\beta||_{L^2(\pa \Omega)}^2 \leq C \big(||\nabla \beta||_{L^2(\Omega)}^2
   + K ||\beta||_{L^2(\Omega)}^2\big).
   \label{basictrace}
  \end{equation}
\end{lemma}

The following estimates are based on \eqref{pwlem} and will be
used to control both $p$ and $D_t p$ in the interior. The first estimate
for $r \geq 2$
is (5.28) in \cite{Christodoulou2000}, while the second estimate
follows from (5.20) in \cite{Christodoulou2000}
 and the estimate $||\nabla q||_{L^2(\Omega)}
\leq ||\Delta q||_{L^2(\Omega)}$ if $q = 0$ on $\pa \Omega$
which is just integration by parts.
\begin{lemma}
 If \eqref{geom} holds and $r \geq 2$, then:
 \begin{equation}
  ||\nabla^r q||_{L^2(\pa \Omega)} + ||\nabla^r q||_{L^2(\Omega)}
  \leq C ||\Pi \nabla^r q||_{L^2(\pa \Omega)} +
  C(K, \Vol(\Omega)) \sum_{s \leq r-1} ||\nabla^s \Delta q||_{L^2(\Omega)}.
  \label{basicproj}
 \end{equation}
 If $q = 0$ on $\pa \Omega$, then in addition:
 \begin{align}
  ||\nabla q||_{L^2(\pa \Omega)}
  + ||\nabla q||_{L^2(\Omega)} \leq C(K) ||\Delta q||_{L^2(\Omega)}
  \label{basicproj2}
 \end{align}
\end{lemma}

We also need estimates to control $\Pi\nabla^r  q$ when $q = 0$ on
$\pa \Omega$. Note that when $r = 2$, we have:
\begin{equation}
  0 = \Pi^{ij}\nabla_i \big(\Pi^{k\ell}\nabla_k q\big)
  = \Pi^{ij}\Pi^{k\ell} \nabla_i \nabla_k q + \Pi^{ij}(\nabla_i\Pi^{k\ell})\nabla_k q.
 \label{calc}
\end{equation}
Recalling that $\Pi^{k\ell} = \delta^{k\ell}  - N^k N^\ell$ and that
since $q = 0$ on $\pa \Omega$, $\nabla_k q = N_k (\nabla_N q)$, so the
above implies that:
\begin{equation}
 \Pi^{ij}\Pi^{k\ell} \nabla_i\nabla_k q = \theta_{ij} \nabla_N q.
 \label{nabla2}
\end{equation}

We will need a higher-order version of this formula.
As explained in \cite{Christodoulou2000}, the idea is that since $q = 0$ on $\pa \Omega$,
$q / d$ is smooth up to the boundary (recall that $d$ is the distance to
the boundary), and we can write:
\begin{equation}
 \Pi \nabla^r q = \Pi \nabla^r(d\, q/d)
 = \sum_{\ell = 0}^r \Pi  (\nabla^\ell d)\otimes (\nabla^{r-\ell} q/d).
 \label{exp}
\end{equation}
On $\pa \Omega$, $d
= \Pi \nabla d = 0$ and
$\Pi \nabla^2 d = \theta$, so that $\Pi \nabla^\ell d \sim \tn^{\ell-2}\theta$, to
highest order.
We now want to replace the derivatives $\nabla^{r-\ell}$ with tangential derivatives
in the above expression.
Writing
$\nabla^i = \tn^i + N^i N_j\nabla^j,$
and inserting this into \eqref{exp} generates derivatives of $N$. Repeatedly
using that $\nabla_N N = 0$, these can be turned into derivatives
of $\theta$.
Also, we have $q/d \sim \nabla_N q$,
and so we should expect:
\begin{equation}
 \Pi \nabla^r q \sim \sum_{\ell = 2}^r (\tn^{\ell-2}\theta)
 \otimes (\nabla^{r-\ell} \nabla_N q).
 \label{heur}
\end{equation}
Note that this also gives a bound for $\tn^{r-2} \theta$, assuming
that $|\nabla_N q|$ is bounded below.
In particular, a direct calculation similar to \eqref{calc} (see (4.21) in
\cite{Christodoulou2000}) also shows that:
\begin{align}
 \tn \theta = (\nabla_N q)^{-1} \bigg(\Pi \nabla^3q
 - 3 \theta \widetilde{\otimes} \tn \nabla_N q\bigg),
 \label{tn1}
\end{align}
where $\widetilde{\otimes}$ is a symmetrization over some of the indices
appearing in the tensor product $\theta \otimes \tn \nabla_N q$ but
for our purposes the exact form is not important.

The precise version of \eqref{heur} and a version of
\eqref{tn1} for higher derivatives of $\theta$ can be found in Proposition 5.9 in \cite{Christodoulou2000}
(see also Proposition 4.3 there):
\begin{prop}
 We have:
 \begin{multline}
  ||\Pi \nabla^r q||_{L^2(\pa \Omega)} \leq
  C(K)\bigg( ||\tn^{r-2} \theta||_{L^2(\pa \Omega)}
  ||\nabla_N q||_{L^\infty(\pa \Omega)}
  + \sum_{k = 1}^{r-1} ||\theta||^k_{L^\infty(\pa \Omega)}
  ||\nabla^{r-k} q||_{L^2(\pa \Omega)}\bigg)\\
  + \big( ||\theta||_{L^\infty(\pa \Omega)}
  + \sum_{k \leq r-2} ||\tn^k \theta||_{L^2(\pa \Omega)}\big)
  \sum_{k \leq r-2} ||\nabla^k q||_{L^2(\pa \Omega)},
  \bigg)
  \label{projest}
 \end{multline}
 and:
 \begin{multline}
  ||\nabla^{r-1} q||_{L^2(\pa \Omega)}
  \leq C\bigg( ||\tn^{r-3} \theta||_{L^2(\pa \Omega)}||\nabla_N q||_{L^2(\pa\Omega)}
  + ||\nabla^{r-2} \Delta q||_{L^2(\Omega)}
  \\
  +C(K, \Vol(\Omega), ||\theta||_{L^2(\pa \Omega)},...,
  ||\tn^{r-4} \theta||_{L^2(\pa \Omega)})
  \bigg( ||\nabla_N q||_{L^\infty(\pa \Omega)}
  + \sum_{s \leq r-3} ||\nabla^s \Delta q||_{L^2(\pa \Omega)}\bigg).
  \label{trace}
 \end{multline}
 In addition, if $|\nabla_N q| \geq  \ve > 0$ and $|\nabla_N q| \geq 2\ve
  ||\nabla_N q||_{L^\infty(\pa \Omega)}$ on $\pa \Omega$ for some $\ve > 0$ then:
 \begin{align}
  ||\tn^{r-2} \theta||_{L^2(\pa \Omega)}^2
  \leq C(\ve^{-1}, K)\bigg( ||\Pi \nabla^r q||_{L^2(\pa \Omega)}
  + \big(||\theta||_{L^\infty(\pa \Omega)}
  + \sum_{k \leq r-3} ||\tn^k \theta||_{L^2(\pa \Omega)}\big)
  \sum_{k \leq r-1} ||\nabla^k  q||_{L^2(\pa \Omega)}\bigg).
  \label{thetabd}
 \end{align}
\end{prop}

The next set of estimates are more well-known.
For example the first estimate is Theorem 9.14 in \cite{gilbarg2001}
combined with Lemma \ref{bdyreg},
and the second estimate follows
from a modification of the proof of Theorem 6.4.8 \cite{Morrey2008}
(see Lemma 2 in \cite{Dolzmann1995} for a concise proof). The third
estimate is not mentioned explicitly in these references but it is clear
that the arguments used to prove \eqref{divest} can be used to prove
\eqref{curlest} as well.
\begin{prop}
  If \eqref{geom} holds, then for any $q \in
  W^{2,p}(\Omega) \cap W_0^{1,p}(\Omega)$, $1 < p <\infty$:
  \begin{equation}
   ||q||_{L^p(\Omega)} +
   ||\nabla q||_{L^p(\Omega)}
   + ||\nabla^2 q||_{L^p(\Omega)}
    \leq  C(K) ||\Delta q||_{L^p(\Omega)}.
   \label{lpest}
  \end{equation}
  If $\Delta q = \div F + g$ for a vector field
  $F$ and a function $g$, and $q \in W^{1,p}_0(\Omega)$, then:
  \begin{equation}
   ||\nabla q||_{L^p(\Omega)} \leq C(K) \big( ||F||_{L^p(\Omega)} +
   ||g||_{L^1(\Omega)}\big).
   \label{divest}
  \end{equation}
  Similarly, if $\beta$ is a vector field and $\Delta \beta = \curl \gamma
  +\rho$ for vector fields $\gamma,\rho$ and $\beta \in W^{1,p}_0(\Omega)$,
  then:
  \begin{equation}
   ||\nabla \beta||_{L^p(\Omega)} \leq C(K)
   \big( || \gamma ||_{L^p(\Omega)} + ||\rho||_{L^1(\Omega}\big).
   \label{curlest}
  \end{equation}
\end{prop}

We will also need the following
``Bernstein'' maximum principle, which is just the maximum
principle for subharmonic functions combined with the fact that
if $\Delta f = 0$, then $\Delta |\nabla f|^2 = 2|\nabla^2 f|^2 \geq 0$:
\begin{prop}
 Suppose $f \in C^3(\Omega)$ is harmonic. Then
 \begin{equation}
  ||\nabla f||_{L^\infty(\Omega)} \leq ||\nabla f||_{L^\infty(\pa \Omega)}.
  \label{bernstein}
 \end{equation}
\end{prop}

We now note the following simple consequence of \eqref{curlest} and \eqref{bernstein}.
Writing $u = u_0 + u_1$ as in \eqref{u0def}-\eqref{u0def2} and
using H\"{o}lder's inequality implies that:
\begin{align}
 ||\nabla u||_{L^p(\Omega)} &\leq C(K, \Vol(\Omega) )\big( ||\omega||_{L^\infty(\Omega)}
 + ||\tn U||_{L^\infty(\pa \Omega)} + ||\dn U||_{L^\infty(\Omega)}\big)\\
 &\leq C(K, \Vol (\Omega)) \A,
 \label{lplinfty}
\end{align}
with $\A$ defined by \eqref{Adef}.

\section{A nonlinear estimate}
\label{nlinsec}

If $v \in L^2(\R^3)$
is a vector
field with $\div v = 0$ and $\curl v = \omega$, then
$\Delta v = \curl \omega$ and using the fact that the Riesz transform
is bounded on $L^p(\R^3)$ for $1 < p <\infty$ leads to:
\begin{equation}
 ||\nabla v||_{L^p(\R^3)} \leq C ||\omega||_{L^p(\R^3)}.
 \label{}
\end{equation}
Unfortunately, such an estimate fails when $ p = \infty$. The
key idea in \cite{Beale1984} is that there is a replacement for
the $p = \infty$ case if we are willing to control more derivatives
of $v$:
\begin{equation}
 ||v||_{L^\infty(\R^3)} \leq C \big((1 + \log^+ ||v||_{H^{r}(\R^3)}) ||\omega||_{L^\infty(\R^2)}
 + ||\omega||_{L^2(\R^3)}\big),
 \label{}
\end{equation}
for any $r > 3/2$, with $\log^+(s) = \max(0, \log(s))$.

There is a version of this inequality that holds on our domain
$\D_t$ as well, provided that we have control over the boundary:
\begin{theorem}
  \label{nlinprop}
  Suppose that \eqref{geom} holds.
  If $\Delta v = \curl \alpha$ in $\D_t$ and $v = 0$ on $\pa \D_t$, then for any $s > 3$:
 \begin{equation}
  ||\nabla v||_{L^\infty(\D_t)} \leq C(K) \bigg(
  \big(1 + \log^+ ||\alpha||_{H^{s-1}(\D_t)}\big)
  ||\alpha ||_{L^\infty(\D_t)} + 1\bigg),
  \label{}
 \end{equation}
 where $\log^+(q) = \max(0, \log(q))$.
\end{theorem}

As in \cite{Beale1984} and \cite{Ferrari1993}, the proof relies on
estimates for the Green's functions for the Laplacian.
We will use the following result, which follows
from Theorems 1.1, 1.3, and 3.3 in \cite{GW1982}. The
result in \cite{GW1982} assumes that the domain satisfies the
uniform exterior sphere condition, which holds if \eqref{geom} holds, by Lemma
\ref{bdyreg}.

\begin{theorem}
  If $\D_t$ satisfies \eqref{geom}, there is a unique function
 $G: \D_t \times \D_t \to \R \cup \{\infty\}$ so that $G \geq 0$,
 $G(x,y) = G(y,x)$ and so that the following properties hold:
 \begin{itemize}
  \item If $y \in \D_t$, then for any $r > 0$:
 \begin{equation}
  G(\cdot, y) \in W^{2,1}(\D_t \setminus B(r, y)) \cap H^1_0(\D_t),
  \label{}
 \end{equation}
 \item for each $\varphi \in C_c^\infty(\Omega)$ we have:
 \begin{equation}
  \int_{\D_t} G(x,y) \Delta \varphi(y) =
  - \int_{\D_t} \delta^{ij} \nabla_{y^i} G(x,y) \nabla_{y^j}  \varphi(y)\, dy = \varphi(x),
  \label{repform1}
 \end{equation}
and
 \item
 if the bound \eqref{geom} holds, then
 for all $x,y \in \D_t$ with $x \not= y$:
  \begin{align}
   |G(x, y)| &\leq C |x-y|^{-1},\label{pw1}\\
   |\nabla G(x,y)| &\leq C(K) |x-y|^{-2},\label{pw2}\\
   |\nabla_x \nabla_y G(x,y)| &\leq C(K) |x-y|^{-3}.
   \label{pw3}
  \end{align}
 \end{itemize}
\end{theorem}

Using this result, we can now provide the
\begin{proof}[Proof of Theorem \ref{nlinprop}]

We argue nearly exactly as in \cite{Beale1984}.
Fix $x \in \D_t$ and let $\delta > 0$ be a parameter to be determined
later. Let
$\chi_\delta$ be a smooth function so that $\chi_\delta(z) = 1$ when
$|z-x| < \delta$, $\chi_\delta(z) = 0$ when $|z-x| \geq 2\delta$,
and so that $|\nabla \chi_\delta| \leq \frac{C}{\delta}$.
By \eqref{repform1} we have:
\begin{equation}
 v_i(x) = \int_{\D_t} G(x,y)\curl \alpha_i(y)\, dy
 = \int_{\D_t} G(x,y) (1-\chi_\delta(y)) \curl \alpha_i(y)\, dy
 + \int_{\D_t} G(x,y) \chi_\delta(y) \curl \alpha_i(y)\, dy,
 \label{repform}
\end{equation}
We now re-write \eqref{repform} as:
\begin{align}
 v_i(x) &= \int_{\D_t}\K_i^j(x,y) (1- \chi_\delta(y)) \alpha_j(y)\, dy
 +\int_{\D_t} G(x,y) \delta_{i\ell} \epsilon^{\ell jk}(\nabla_k \chi_\delta(y))\alpha_j(y)\, dy
 + \int_{\D_t} G(x,y) \chi_\delta(y) \curl \alpha_i(y)\, dy\\
 &\equiv v_A(x) + v_B(x) + v_C(x),
 \label{}
\end{align}
where $\epsilon$ is the fully anti-symmetric symbol normalized
by $\epsilon^{123} = 1$ and $\K^j_i(x,y) = \delta_{i\ell} \epsilon^{j \ell k} \nabla_{y^k}G(x,y)$.

Using $|\nabla_x \K(x,y)| \leq |\nabla_x \nabla_y G(x,y)| \leq C(K) |x-y|^{-3}$,
we have:
\begin{align}
 |\nabla v_A(x)| &\leq C(K) ||\alpha||_{L^\infty(\D_t)}
 \bigg(
 \int_{ \{ \delta \leq |x-y| \leq  1\}\cap \D_t} |x-y|^{-1} dy
 +\int_{ \{ 1\leq |x-y|\}\cap \D_t} |x-y|^{-1} dy\bigg)\\
 &\leq C(K) \bigg( C_1 ||\alpha||_{L^\infty(\D_t)} - \log\delta ||\alpha||_{L^\infty(\D_t)}\bigg),
\end{align}
where $C_1 = C_1(\Vol \D_t)$.

Since $|\nabla \chi_\delta| \leq C\delta^{-1}$ and $\nabla \chi_\delta$ is
supported on the annulus $\delta \leq |x-y| \leq 2\delta$, and since
$|\nabla G(x,y)| \leq C(K) |x-y|^{-2}$ we have:
\begin{equation}
 |\nabla v_B(x)| \leq C(K) \delta^{-1} ||\alpha||_{L^\infty(\D_t)}
 \int_{\delta \leq |x-y|\leq 2\delta} |x-y|^{-2} \, dy\leq
 C(K)||\alpha||_{L^\infty(\D_t)}.
 \label{}
\end{equation}

Finally, to control $\nabla v_C(x)$, we use
H\"{o}lder's inequality and then Sobolev's inequality:
\begin{align}
  |\nabla v_C(x)| \leq \bigg|\int_{D_t} \nabla_x G(x,y) \chi_\delta(y)
  \nabla_{y} \alpha(y) dy\bigg|
  &\leq C(K)\bigg( \int_{\{0\leq |x-y| \leq \delta\}} |x-y|^{-12/5} \, dy
  \bigg)^{5/6}
  ||\nabla \alpha||_{L^6(\D_t)}\\
  &\leq C(K) \delta^{1/2}||\alpha||_{H^2(\D_t)}.
 \label{}
\end{align}
Therefore we have shown:
\begin{equation}
|\nabla v (x)| \leq C(K, \Vol \D_t)\big( (1 - \log\delta)
||\alpha||_{L^\infty(\D_t)} + \delta^{1/2} ||\alpha||_{H^2(\D_t)}\big).
\label{}
\end{equation}
We now take $\delta^{-1/2} = \min(1, ||\alpha||_{H^2(\D_t)})$, so that the
above becomes:
\begin{equation}
|\nabla v(x)| \leq C(K, \Vol \D_t) \big( (1 + \log||\alpha||_{H^2(\D_t)})||\alpha||_{L^\infty(\D_t)}
+1\big)
\label{}
\end{equation}
and noting that by \eqref{mass}, $\Vol \D_t = \Vol\D_0$, this
completes the proof.
\end{proof}
\section{Energy estimates}
\label{enestsec}

The results in this section are nearly identical to those in \cite{Christodoulou2000}.
The only difference is that we are slightly more explicit about the lower-order
terms that occur in the computations.

We define:
\begin{equation}
 E_r(t) = \int_{\Omega} \delta^{ij}\Pi^{IJ}(\nabla_I^r u_i) (\nabla_J^r u_j)\, dx
 + \int_{\pa \Omega} \Pi^{IJ}(\nabla_I^r p) (\nabla_J^r p)  |\nabla p|^{-1}
 dS,
 \label{endef}
\end{equation}
as well as:
\begin{equation}
 K_r(t) = \int_{\D_t} |\nabla^{r-1} \omega|^2 \, dx,
 \label{}
\end{equation}
and:
\begin{equation}
 \E_r(t) = E_r(t) + K_r(t).
 \label{}
\end{equation}
We also write:
\begin{align}
 \E_0(t) = \int_{\D_t} |u|^2\,dx,
 \label{}
\end{align}
and a simple calculation shows that $\E_0$ is conserved.
We will estimate these quantites assuming that the following bounds hold
on $\pa\D_t$:
\begin{align}
 |\theta| + \frac{1}{\iota_0} &\leq K,\label{kassump}\\
 |\nabla p| &\geq \delta > 0
 \label{deltassump}
\end{align}

Our estimates will involve the following quantity:
\begin{equation}
 \A =
 ||\omega||_{L^\infty(\D_t)} +
 ||\nabla u||_{L^\infty(\pa \D_t)} +
 ||\dn U||_{L^\infty(\pa \D_t)},
 \label{Adef}
\end{equation}
where $U = u|_{\pa \D_t}$.
We then have the following energy estimates:
\begin{prop}
  \label{enestprop}
  If the assumptions \eqref{kassump}-\eqref{deltassump} hold,
  then for $r = 1,2,3$:
  \begin{equation}
   \frac{d}{dt} \E_r \leq
   C(K, \delta^{-1})
   \big(||\nabla u||_{L^\infty(\D_t)}
   +
 ||\nabla_N D_t p||_{L^\infty(\pa \D_t)}
   + \A
   + \A^2 \big) \,
   \sum_{r = 0}^3 \E_r
   \label{enlow}
  \end{equation}
  and for $r \geq 4$,
  there is a polynomial $P$ so that:
  \begin{equation}
   \frac{d}{dt} \E_r
   \leq C(K, \delta^{-1})
   \big(||\nabla u||_{L^\infty(\D_t)} +
   ||\nabla_N D_t p||_{L^\infty(\pa \D_t)} + \A
   + \A^2\big)\, \E_r\, P(\E_{r-1},...,
   \E_0).
   \label{enhigh}
  \end{equation}
\end{prop}

The proof of this estimate is nearly identical to the proof of
Theorem 7.1 in
in \cite{Christodoulou2000}, except that we need to ensure that the dependence on
$||\nabla u||_{L^\infty}$ is linear (compare with (7.16) in \cite{Christodoulou2000}).
The only
part of the argument that needs to be changed is the proof of the estimates for
derivatives of $D_t p$ on $\pa \D_t$. This is because we will ultimately bound
derivatives of $D_tp$ by derivatives of $\Delta D_tp$ and this is cubic
in $\nabla u$.

Before proving the above estimates, it is helpful to see what quantities
the energies bound. The following lemma is Lemma 7.3 in \cite{Christodoulou2000},
and relies on the elliptic estimates described in section \ref{ellsec}.
We need slightly different estimates to deal with the pressure and second fundamental
form depending on how many derivatives are present, because in the estimates
for $\E_3$,
we need all of our estimates to be linear in $||\nabla u||_{L^\infty(\Omega)}$.
\begin{lemma}
  \label{coerlem}
  We have:
  \begin{align}
   ||\nabla^r u||_{L^2(\Omega)}^2 &\leq C \E_r,
   && ||\Pi \nabla^r p||_{L^2(\pa \Omega)}^2
   \leq ||\nabla p||_{L^\infty(\pa \Omega)} \E_r,
   \label{ellu}
 \end{align}
 \begin{align}
  ||\nabla p||_{L^2(\pa \Omega)}^2 +
  ||\nabla^2 p||_{L^2(\pa\Omega)}^2
  \leq C(K, \Vol \Omega) \big(
  ||\nabla p||_{L^\infty(\pa \Omega)} + ||\nabla u||_{L^\infty(\Omega)}^2\big)
   \sum_{s = 0}^2
  \E_s
  \label{ellplow}
 \end{align}
 and:
  \begin{align}
   ||\nabla^r p||_{L^2(\pa \Omega)}^2
   + ||\nabla^r p||_{L^2(\Omega)}^2
  & \leq C(K, \Vol \Omega) \big( ||\nabla p||_{L^\infty(\pa \Omega)}
   + ||\nabla u||_{L^\infty(\Omega)}^2 \big)
   \sum_{s = 0}^r \E_s
   \label{ellp}
  \end{align}
  If the bound \eqref{deltassump} holds then:
  \begin{align}
   ||\theta||_{L^2(\pa \Omega)}^2 \leq C(\delta^{-1}) \E_2, &&
   ||\tn \theta||_{L^2(\pa \Omega)}^2 \leq
   C(K, \delta^{-1}, \Vol \Omega)\bigg(\E_3 + ||\nabla u||_{L^\infty(\D_t)}
   \sum_{s = 0}^{2} \E_s\bigg),
   \label{thetaless}
  \end{align}
  and for $r \geq 4$, there is a polynomial $P$ so that:
  \begin{equation}
   ||\tn^{r-2} \theta||_{L^2(\pa \Omega)}^2
   \leq C(K, \delta^{-1}) (1 + ||\nabla p||_{L^\infty(\pa \Omega)}
   + ||\nabla u||_{L^\infty(\Omega)}^2) P(\E_{r-1},...\E_0) \E_r.
   \label{thetaannoying2}
  \end{equation}
\end{lemma}
\begin{proof}
  The first estimate follows from \eqref{pwlem} and
  the second estimate follows from
  the definition of $\E_r$.
  The estimates in \eqref{ellplow} follow from the inequality
  \eqref{basicproj2} along with \eqref{nabla2} and \eqref{prespois},
  while the estimate \eqref{ellp} follows from \eqref{basicproj}
  \eqref{projest}.

  The first estimate in \eqref{thetaless} is just \eqref{nabla2} and the
  second estimate is \eqref{tn1} combined with the estimates
  \eqref{ellplow}.
   To prove the estimate \eqref{thetaannoying2},
   we combine
  \eqref{thetabd}, \eqref{ellp} and use the assumption
  \eqref{deltassump}:
  \begin{multline}
   ||\tn^{r-2} \theta||_{L^2(\pa \Omega)}^2
   \leq C(K,\delta^{-1})\big(1 + ||\nabla p||_{L^\infty(\pa \Omega)}
   + ||\nabla u||_{L^\infty(\Omega)}^2 \big)
   \bigg( ||\theta||_{L^\infty(\pa \Omega)}
   + \sum_{k \leq r-3} ||\tn^k \theta||_{L^2(\pa \Omega)}\bigg)
   \sum_{s = 0}^r \E_s
   \label{}
  \end{multline}
\end{proof}

Taking the divergence of \eqref{mom} and using the fact that
$[D_t, \nabla_i] = -(\nabla_i u^j)\nabla_j$, we have:
\begin{equation}
 \Delta p = -(\nabla_i u^j)(\nabla_j u^i).
 \label{prespois}
\end{equation}
Applying $D_t$ to both sides of this equation, a
calculation using \eqref{prespois} and
\eqref{commform} (see just below (6.14) in \cite{Christodoulou2000}) yields:
\begin{equation}
 \Delta D_t p = p_1 + p_2 + p_3,
 \label{dtprespois}
\end{equation}
with:
\begin{align}
 p_1 = 4g^{ab} g^{cd}(\nabla_a u_c) \nabla_b\nabla_d p, &&
 p_2 = 2 (\nabla_a u^d)(\nabla_d u^c)(\nabla_c u^a),&&
 p_3 = -(\Delta u^e)\nabla_e p.
 \label{}
\end{align}
Differentiating
\eqref{mom}, \eqref{prespois} and \eqref{dtprespois} gives
Lemma 6.1 of \cite{Christodoulou2000}:
\begin{lemma}
  \begin{align}
   |D_t \nabla^r u + \nabla^{r+1} p|
   + |D_t \nabla^{r-1} \curl u| + |\nabla^{r-1} \Delta p|
   &\leq C \sum_{s = 0}^{r-1}
   |(\nabla^{1+ s} u) (\nabla^{r-s} u)|, \label{hotdeul}\\
   |\Pi(D_t \nabla^{r}p + (\nabla^r u)\cdot \nabla p
   - \nabla^{r}D_t p)|
   &\leq C \sum_{s = 0}^{r-2} |\Pi \big( (\nabla^{1+s}u)\cdot
   \nabla^{r-s} p\big)|,
   \label{hotpbdy}
  \end{align}
  and
  \begin{multline}
   |\nabla^{r-2} \Delta D_t p - (\nabla^{r-2} \Delta u)\cdot \nabla p|
   \\
   \leq C \sum_{s = 0}^{r-2} |(\nabla^{1+s} u)(\nabla^{r-s} p)|
   + C \sum_{r_1 + r_2 + r_3 = r-2}
   |(\nabla^{1+r_1}u)(\nabla^{1+r_2}u)(\nabla^{1+r_3}u)|
   \label{hotdtp}
  \end{multline}

\end{lemma}

The next ingredient we will need are the following
$L^2$ estimates for $\Delta D_t p$. These are similar to the estimates
in \cite{Christodoulou2000} except that we need to ensure that
$||\nabla u||_{L^\infty(\Omega)}$ appears
with the same homogeneity as $\nabla^{r-2}\Delta D_t p$:
\begin{lemma}
  \label{dtpests}
  For $r = 2, 3$:
  \begin{align}
   ||\nabla^{r-2} \Delta D_t p||_{L^2(\Omega)}^2
   + &\leq C(K, \Vol \Omega)
   ||\nabla u||_{L^\infty(\Omega)}^2 \A^2 \sum_{k = 0}^{r} \E_k,
   \label{dtplow}
  \end{align}
 and for $r \geq 4$:
 \begin{equation}
  ||\nabla^{r-2} \Delta D_t p||_{L^2(\Omega)}^2
  \leq C(K, Vol(\Omega)) ||\nabla u||_{L^\infty(\Omega)}^2
  \E_r \sum_{k = 0}^{r-1} \E_k  \label{dtphigh}
 \end{equation}
\end{lemma}

\begin{proof}
  Using \eqref{hotdtp}, we need to control:
  \begin{align}
   &||(\nabla^{r-2} \Delta u)(\nabla p)||_{L^2(\Omega)},\label{drp1}\\
   &|| (\nabla^{1+s} u)(\nabla^{r-s} p)||_{L^2(\Omega)}, &&
    s = 0,..., r-2\label{drp2}\\
    &||(\nabla^{1+r_1} u) (\nabla^{1+r_2} u) (\nabla^{1+ r_3}u)
    ||_{L^2(\Omega)}, && r_1 + r_2 + r_3 = r-2.
   \label{drp3}
  \end{align}

We first control \eqref{drp1} by $||\nabla^{r} u||_{L^2(\Omega)}
||\nabla p||_{L^\infty(\Omega)}.$
 When $r \leq 3$, we use Sobolev embedding
 \eqref{intsob2} and the elliptic estimate \eqref{lpest} with
 Lebsegue exponent $p = 3$:
 \begin{align}
  ||\nabla p||_{L^\infty(\Omega)}
  &\leq C(K)\big( ||\nabla p||_{L^3(\Omega)} +
  ||\nabla^2 p||_{L^3(\Omega)}\big)\\
  &\leq C(K)||\nabla u||_{L^\infty(\Omega)} ||\nabla u||_{L^3(\Omega)}.
  \label{}
 \end{align}
 By \eqref{lplinfty}, we have $||\nabla u||_{L^3(\Omega)}
 \leq C(K, Vol(\Omega)) \A.$ When $r \geq 4$, we can instead
 use the Sobolev inequality \eqref{intsob2} with $k = p = 2$ and the estimate \eqref{ellp}:
 \begin{align}
  ||\nabla p||_{L^\infty(\Omega)}^2
  &\leq C(K, Vol(\Omega)) \sum_{k = 1}^3 ||\nabla^k p||_{L^2(\Omega)}^2\\
  &\leq C(K, Vol(\Omega)) \big( ||\nabla p||_{L^\infty(\pa \Omega)} +
  ||\nabla u||_{L^\infty(\Omega)}^2\big)\E_3
  \label{yetanother}
 \end{align}

 Next, to control \eqref{drp2}, when $r = 2$ it is bounded by:
 \begin{align}
   ||\nabla u||_{L^\infty(\Omega)} ||\nabla^2 p||_{L^2(\Omega)}
   &\leq C(K) ||\nabla u||_{L^\infty(\Omega)} ||\nabla u||_{L^4(\Omega)}^2
   \leq ||\nabla u||_{L^\infty(\Omega)} \A^2,
  \label{anotherlow}
 \end{align}
 and when $r = 3$ we instead bound it by:
 \begin{align}
  ||\nabla u||_{L^\infty(\Omega)} ||\nabla^2 p ||_{L^2(\Omega)}
  + ||\nabla^2 u||_{L^2(\Omega)} ||\nabla p||_{L^\infty(\Omega)},
  \label{}
 \end{align}
 and then control the first term as in \eqref{anotherlow} and use
 \eqref{yetanother} to bound $||\nabla p||_{L^\infty(\Omega)}$.

For $r \geq 4$,
  we use the interpolation inequality
 \eqref{intinterpu}:
 \begin{equation}
   ||(\nabla^{1+s} u)(\nabla^{r-s} p)||_{L^2(\Omega)}
   \leq ||\nabla u||_{L^\infty(\Omega)} \sum_{k= 1 }^{r}
   ||\nabla^{k} p||_{L^2(\Omega)}
   + ||\nabla p||_{L^\infty(\Omega)}
   \sum_{k = 1}^{r} ||\nabla^k u||_{L^2(\Omega)}.
  \label{}
 \end{equation}
 These terms can be bounded by the right-hand side of \eqref{dtphigh} by arguing as
 above and using \eqref{ellp}.

 Finally we bound \eqref{drp3}. For $r = 2$ we use \eqref{lplinfty}
 and Sobolev embedding \eqref{intsob2}:
 \begin{align}
  ||\nabla u||_{L^\infty(\Omega)} ||\nabla u||^2_{L^4(\Omega)}
  &\leq C(K, \Vol \Omega) ||\nabla u||_{L^\infty(\Omega)} \A
  \big( ||\nabla u||_{L^2(\Omega)} + ||\nabla^2 u||_{L^2(\Omega)}\big)\\
&\leq C(K, \Vol \Omega)||\nabla u||_{L^\infty(\Omega)} \A\,\big( \E_2
+ \E_1\big),
  \label{}
 \end{align}
 and for $r = 3$ the same strategy gives that \eqref{drp3} is bounded
 by:
 \begin{align}
  ||\nabla u||_{L^\infty(\Omega)}
  ||\nabla u||_{L^4(\Omega)} ||\nabla^2 u||_{L^4(\Omega)}
  \leq C(K, \Vol \Omega) ||\nabla u||_{L^\infty(\Omega)} \A \, \big(\E_3
  + \E_2\big).
  \label{}
 \end{align}
 When $r \geq 4$, we use the interpolation inequality \eqref{intinterpu}
 and Sobolev embedding \eqref{intsob2} to bound it by:
 \begin{align}
  ||\nabla u||_{L^\infty(\Omega)}^2 \sum_{k = 0}^{r-1}
  ||\nabla^k u||_{L^2(\Omega)}
  &\leq C(K) ||\nabla u||_{L^\infty(\Omega)} \bigg(
  \sum_{k = 0}^3 ||\nabla^k u||_{L^2(\Omega)}\bigg)
\bigg(\sum_{k = 0}^{r-1}
  ||\nabla^k u||_{L^2(\Omega)}\bigg)\\
  &\leq C(K) ||\nabla u||_{L^\infty(\Omega)} \bigg(\sum_{s = 0}^{r-1} \E_{s}
  \bigg)^2
  \label{}
 \end{align}

\end{proof}

Combining the previous two results, we have:

\begin{cor}
  With $\A$ defined by \eqref{Adef}, for $r = 2,3$:
 \begin{equation}
  ||\Pi \nabla^r D_t p||^2_{L^2(\pa \Omega)}
  + ||\nabla^{r-1} D_t p||^2_{L^2(\pa \Omega)}
  \leq C(K, \Vol(\Omega))(||\nabla u||^2_{L^\infty(\Omega)} + 1)\A
  \E_3
  \label{projlow}
 \end{equation}
 and for $r \geq 4$, there is a polynomial $P$ so that:
 \begin{equation}
  ||\Pi \nabla^r D_t p||^2_{L^2(\pa \Omega)}
  + ||\nabla^{r-1} D_t p||^2_{L^2(\pa \Omega)}
  \leq C(K, \Vol(\Omega))(||\nabla u||^2_{L^\infty(\Omega)} + 1)\A
  \,\bigg(\E_r + P(\E_{r-1})\bigg)
  \label{projhigh}
 \end{equation}
\end{cor}
\begin{proof}

  The $r = 2$ case follows by first applying \eqref{trace}
and \eqref{dtplow}:
\begin{align}
 ||\nabla D_t p||_{L^2(\pa \Omega)}^2
 \leq C(K, \Vol(\Omega)) ||\Delta D_t p||_{L^2(\Omega)}^2
 &\leq C(K, \Vol(\Omega)) ||\nabla u||_{L^\infty(\Omega)}^2 \big(||\omega||_{L^\infty(\Omega)}^2
 + ||\tn u||^2_{L^\infty(\pa \Omega)}\big)\E_3\\
 &\leq C(K, \Vol(\Omega))||\nabla u||_{L^\infty(\Omega)}
 \A^2 \, \E_3,
 \label{r21}
\end{align}
and then applying \eqref{nabla2}:
\begin{align}
 ||\Pi \nabla^2 D_t p||_{L^2(\pa \Omega)}
 &\leq  ||\theta||_{L^2(\pa \Omega)} ||\nabla_N D_t p||_{L^\infty(\pa \Omega)}.
 \label{r22}
\end{align}

When $r = 3$, we have:
\begin{align}
 ||\nabla^2 D_t p||_{L^2(\pa \Omega)}^2
 &\leq C(K, \Vol(\Omega))\bigg( ||\Pi \nabla^2 D_t p||_{L^2(\pa \Omega)}^2
 + ||\Delta D_t p||_{L^2(\Omega)}^2 + ||\nabla \Delta D_t p||_{L^2(\Omega)}^2\bigg),
 \label{}
\end{align}
and using \eqref{r22}, \eqref{r21} and \eqref{dtplow}, the right-hand side is bounded
by \eqref{projlow}.

Using \eqref{projest}, we also have:
\begin{align}
 ||\Pi \nabla^3 D_t p||_{L^2(\pa \Omega)}
 &\leq C(K) \bigg(||\tn \theta||_{L^2(\pa \Omega)} ||\nabla_N D_t p||_{L^\infty(\pa \Omega)}
 + ||\nabla D_t p||_{L^2(\pa \Omega)} + ||\nabla^2 D_t p||_{L^2(\pa \Omega)}\\
 &+ \big( ||\tn \theta||_{L^2(\pa \Omega)} + ||\theta||_{L^2(\pa \Omega)}
 + ||\theta||_{L^\infty(\pa \Omega)}\big) ||\nabla D_t p||_{L^\infty(\pa \Omega)}\bigg),
 \label{}
\end{align}
Using \eqref{nabla2} and \eqref{tn1} with $q = p$ and then
\eqref{ellp},
this is bounded by the right-hand side
of \eqref{projhigh}.

For $r \geq 4$, the argument is the same, except we use the bound \eqref{dtphigh}
in place of the bound \eqref{dtplow}.
\end{proof}

  \begin{proof}[Proof of Proposition \ref{enestprop}]
    By the Reynolds transport theorem,
    \begin{equation}
    \frac{1}{2} \frac{d}{dt} \int_{\D_t}
     |\nabla^{r-1} \omega(t)|^2 \, dx
     = \int_{\D_t} D_t \nabla^{r-1} \omega \cdot
     \nabla^{r-1} \omega \, dx.
     \label{}
    \end{equation}
    By \eqref{hotdeul} and the interpolation inequality \eqref{intinterpu}, this gives:
    \begin{equation}
      \frac{d}{dt} K_r(t) \leq
      C(K) ||\nabla u||_{L^\infty(\D_t)} \E_r.
     \label{}
    \end{equation}

    To control the time derivative of $E_r$ (defined in \eqref{endef}),
    we use Proposition 5.12 of \cite{Christodoulou2000} with $\alpha = \nabla^r p$ and
    $\beta = \nabla^{r-1} u$ and $\nu = (-\nabla p)^{-1}$.
    This gives:
    \begin{align}
     \frac{d}{dt} E_r(t)
     &\leq C \sqrt{E_r}\big(
     ||\Pi (D_t \nabla^r p - (\nabla p) N^k \nabla^{r} u_k)||_{L^2(\pa \Omega)}
     + ||D_t \nabla^r u + \nabla \nabla^{r} p||_{L^2(\Omega)}\big)\\
     &+ CK E_r + C\sqrt{E_r}||\curl \nabla^{r-1} u||_{L^2(\Omega)}
     + \big(K ||\nabla^r p||_{L^2(\Omega)} + ||\curl \nabla^{r-1} u||_{L^2(\Omega)}\big)^2.
     \label{dtE}
    \end{align}
    By Lemma \ref{coerlem}, the terms on the second line are all bounded by
    $C(K,\Vol(\Omega)) \A \E_r$.
    Also, using \eqref{hotdeul}, the interpolation inequality \eqref{intinterpu}
    and Lemma \ref{coerlem}, the second term in \eqref{dtE} is bounded
    by the right-hand side of \eqref{enlow} (resp. \eqref{enhigh}).
It remains to control the first term in \eqref{dtE}.
Using \eqref{hotpbdy}:
    \begin{equation}
     || \Pi (D_t \nabla^r p + (\nabla^r u)\cdot \nabla p)||_{L^2(\pa \Omega)}
     \leq C\bigg(||\Pi \nabla^r D_t p||_{L^2(\pa \Omega)}
     + \sum_{s = 0}^{r-2} ||\Pi\big( (\nabla^{1+s} u)\cdot
     \nabla^{r-s} p\big)||_{L^2(\pa \Omega)}\bigg).
     \label{}
    \end{equation}
    By \eqref{projlow} (resp. \eqref{projhigh}), the first term here is bounded
    by the right-hand side of \eqref{enlow} (resp. \eqref{enhigh}).

    To control $||\Pi (\nabla^{1+s} u) \cdot
    (\nabla^{r-s} p)||_{L^2(\pa \Omega)}$ for $s = 0,..., r-2,$
     we note that when $r = 2$ the result is bounded by:
     \begin{align}
       ||\nabla u||_{L^\infty(\pa \Omega)} ||\nabla^2 p||_{L^\infty(\pa \Omega)},
      \label{bdyinterploww}
     \end{align}
     and for $r = 3$, the result is bounded by:
    \begin{equation}
     || \nabla u||_{L^\infty(\pa \Omega)} ||\nabla^3 p||_{L^2(\pa \Omega)}
     + ||\nabla^2 u||_{L^2(\pa \Omega)} ||\nabla^2 p||_{L^\infty(\pa \Omega)}.
     \label{bdyinterplow}
    \end{equation}
    By the pointwise estimate \eqref{pwlem}, the equation \eqref{prespois}
    and \eqref{nabla2}, we have:
    \begin{align}
     ||\nabla^2 p ||_{L^\infty(\pa \Omega)}
     \leq C \big(||\nabla u||_{L^\infty(\pa \Omega)}^2 +
     ||\theta||_{L^\infty(\pa \Omega)} ||\nabla_N p||_{L^\infty(\pa \Omega)}\big),
     \label{}
    \end{align}
    and combining this with the trace inequality \eqref{basictrace} and Lemma
    \ref{coerlem} shows that \eqref{bdyinterploww} and
    \eqref{bdyinterplow} are controlled by the
    right-hand side of \eqref{enlow}.

    To control $||\Pi (\nabla^{1+s} u)\cdot \nabla^{r-s}p)||_{L^2(\pa \Omega)}$
    for $r \geq 4$, we note that we could use Sobolev embedding, the trace inequality
    \eqref{basictrace} and Lemma \ref{coerlem} to show that this is bounded
    by interior terms, but this would lead to estimates that are not linear
    in the highest order norm $\E_r$. The idea from \cite{Christodoulou2000}
    is to use the fact that because
    of the prescence of the projection $\Pi$, the derivatives $\nabla^{1+s},
    \nabla^{r-s}$ are nearly tangential derivatives and for tangential
    derivatives we can use the interpolation inequality \eqref{bdyinterp} to
    control the intermediate terms. Arguing
    as in (7.25) in \cite{Christodoulou2000} this gives:
\begin{align}
 || \Pi &\big( (\nabla^{1+s} u) \cdot (\nabla^{r-s} p) \big)||_{L^2(\pa \Omega)}\\
 &\leq C(K)\bigg( ||\nabla u||_{L^\infty(\pa \Omega)} + \sum_{k = 0}^{r-2}
 ||\nabla^k u||_{L^2(\pa \Omega)}\bigg)
 ||\nabla^r p||_{L^2(\pa \Omega)}\\
 &+ C(K)\bigg( ||\nabla^2 p||_{L^\infty(\pa \Omega)} +
 \sum_{k = 0}^{r-1} ||\nabla^k p||_{L^2(\pa \Omega)}\bigg)
 ||\nabla^{r-1} u||_{L^2(\pa \Omega)}\\
 &+
 C(K) \bigg( ||\theta||_{L^\infty(\pa \Omega)}
 + \sum_{k = 0}^{r-2} ||\tn^{k} \theta||_{L^2(\pa \Omega)}\bigg)
 \bigg( ||\nabla u||_{L^\infty(\pa \Omega)} + \sum_{k = 0}^{r-2}
 ||\nabla^k u||_{L^2(\pa \Omega)}\bigg)\\
 &\times \bigg( ||\nabla^2 p||_{L^\infty(\pa \Omega)} +
 \sum_{k = 0}^{r-1} ||\nabla^k p||_{L^2(\pa \Omega)}\bigg)
 \label{bdyinterineq}
\end{align}
Using the Sobolev estimates \eqref{bdysob2}, \eqref{intsob2}
along with Lemma \ref{coerlem}
and arguing as above proves \eqref{enhigh}.
  \end{proof}

\appendix
\section{Sobolev Estimates}

Here we collect the various Sobolev embeddings that we will
rely on. These are all well-known, but what is important is that the
constants in the various inequalities depend only on bounds for the second
fundamental form and the volume of $\D_t$ (which is constant if
\eqref{mass} holds).
  The proofs of these theorems with these constants
  appear in the appendix to \cite{Christodoulou2000}.
  \subsection{Interpolation inequalities}
  We will require interpolation inequalities both on $\pa \D_t$
  and $\D_t$.

  \begin{lemma}
    Suppose that:
   \begin{equation}
    \frac{m}{s} = \frac{k}{p} + \frac{m-k}{q},
    2 \leq p \leq s \leq q \leq \infty,
    \label{}
   \end{equation}
   and let $a = k/m$. Then there is a constant $C$ depending only
   on $m$ so that for any $(0,r)$ tensor $\alpha$:
   \begin{equation}
    ||\tn^k \alpha||_{L^s(\pa \D_t)} \leq C||\alpha||_{L^q(\pa \D_t)}^{1-a}
    ||\tn^m \alpha||_{L^p(\pa \D_t)}^a.
    \label{bdyinterp}
   \end{equation}
   In addition,  if $\iota_0 \geq \frac{1}{K}$, then:
  \begin{equation}
     \sum_{j = 0}^k ||\nabla^j \alpha||_{L^s(\D_t)}
     \leq C||\alpha||_{L^q(\D_t)}^{1-a}\bigg(\sum_{i = 0}^m
     ||\tn^i \alpha||_{L^p(\D_t)} K^{m-i}\bigg)^a.
     \label{intinterp}
    \end{equation}
    In particular, if $\ell + m = k$ then:
    \begin{equation}
      ||\tn^\ell \alpha \tn^m \beta||_{L^2(\pa\D_t)}
      \leq C\bigg(||\alpha||_{L^\infty(\pa \D_t)} \sum_{\ell = 0}^k
      ||\tn^\ell \beta||_{L^2(\pa \D_t)}
      + ||\beta||_{L^\infty(\pa \D_t)} \sum_{\ell = 0}^k
      ||\tn^\ell \alpha||_{L^2(\pa \D_t)}\bigg)
      \label{bdyinterpu}
    \end{equation}
    and
    \begin{equation}
      ||\nabla^\ell \alpha \nabla^m \beta||_{L^2(\D_t)}
      \leq C(K)\bigg(||\alpha||_{L^\infty(\D_t)} \sum_{\ell = 0}^k
      ||\nabla^\ell \beta||_{L^2( \D_t)}
      + ||\beta||_{L^\infty(\D_t)} \sum_{\ell = 0}^k
      ||\nabla^\ell \alpha||_{L^2(\D_t)}\bigg)
      \label{intinterpu}
    \end{equation}
  \end{lemma}

  \subsection{Sobolev and Poincar\'e inequalities}
  \begin{lemma}
   Suppose that $1/\iota_0 \leq K$. Then for any $(0,r)$-tensor:
   \begin{align}
     ||\alpha||_{L^{2p/2-kp}(\pa \D_t)}
     &\leq C(K) \sum_{\ell = 0}^k ||\tn^\ell\alpha||_{L^p(\pa \D_t)},
     && 1 \leq p \leq \frac{2}{k},\label{bdysob1}\\
     ||\alpha||_{L^\infty(\pa \D_t)}
     &\leq C(K)\sum_{0 \leq \ell \leq k-1}
     ||\tn^\ell \alpha||_{L^p(\pa \D_t)}, && k > \frac{2}{p},
    \label{bdysob2}
   \end{align}
   and
   \begin{align}
     ||\alpha||_{L^{3p/3-kp}(\D_t)}
     &\leq C(K) \sum_{\ell = 0}^k ||\nabla^\ell\alpha||_{L^p(\D_t)},
     && 1 \leq p \leq \frac{3}{k},\label{intsob1}\\
     ||\alpha||_{L^\infty( \D_t)}
     &\leq C(K)\sum_{0 \leq \ell \leq k-1}
     ||\nabla^\ell \alpha||_{L^p(\D_t)}, && k > \frac{3}{p}.
    \label{intsob2}
   \end{align}
  \end{lemma}
  We will also need the following version of the Poincar\'e inequality, whose
  proof is also in \cite{Christodoulou2000}.
  \begin{lemma}
   If $q = 0$ on $\pa \D_t$ then:
   \begin{align}
    ||q||_{L^2(\D_t)} \leq C (Vol \D_t)^{1/3} ||\nabla q||_{L^2(\D_t)},
    \label{poin}\\
    ||\nabla q||_{L^2(\D_t)} \leq C (Vol \D_t)^{1/6}
    ||\Delta q||_{L^2(\D_t)}\label{poin2}.
   \end{align}
  \end{lemma}

\section*{Acknowledgements}
The author would like to thank Hans Lindblad for suggesting this problem
and for valuable comments on an early version of this manuscript, as well
as Theodore Drivas and Huy Nguyen for many helpful discussions.

%\bibliographystyle{abbrv}
%\bibliography{freebdy-mr}

\end{document}